\pdfoutput=1
\documentclass[11pt,reqno,oneside]{amsart}

\usepackage[letterpaper,hmargin=1.35in,vmargin=1.22in]{geometry}
\usepackage{microtype} 

\usepackage{amssymb}

\usepackage{tikz}
\usetikzlibrary{matrix,arrows,cd,calc}

\usepackage{mathtools} 
\usepackage{xspace}

\usepackage[hyphens]{url}
\usepackage{ifpdf}
\ifpdf  \usepackage[pdftex,bookmarks=false]{hyperref}
\else   \usepackage[dvips,bookmarks=false]{hyperref}
        \usepackage{breakurl} 
\fi
\usepackage{color}
\definecolor{darkgreen}{rgb}{0,0.45,0}
\hypersetup{colorlinks,urlcolor=blue,citecolor=darkgreen,linkcolor=darkgreen,linktocpage}

\numberwithin{equation}{section}

\usepackage[capitalize]{cleveref}
\Crefname{lem}{Lemma}{Lemmas}
\Crefname{prop}{Proposition}{Propositions}

\usepackage[style=alphabetic,maxnames=10,maxalphanames=4,backend=bibtex,hyperref=true]{biblatex}
\theoremstyle{definition}
\newtheorem{defn}{Definition}[section]

\theoremstyle{plain}
\newtheorem{lem}[defn]{Lemma}
\newtheorem{prop}[defn]{Proposition}
\newtheorem{cor}[defn]{Corollary}
\newtheorem{thm}[defn]{Theorem}

\newtheorem{thmintro}{Theorem}

\theoremstyle{remark}
\newtheorem{rmk}[defn]{Remark}
\newtheorem{eg}[defn]{Example}

\newcommand{\define}[1]{\textbf{\boldmath{#1}}}
\newcommand{\dfn}[1]{\textbf{\boldmath{#1}}}
\newcommand{\defeq}{\vcentcolon\equiv}  
\newcommand{\defpath}{:=} 

\gdef\notesoff{\gdef\note##1{\null}}
\notesoff

\newcommand{\sbt}{{\,\begin{picture}(-1,1)(-1,-3)\circle*{3}\end{picture}\ } }

\newcommand{\Grp}{\mathsf{Grp}}
\newcommand{\Ab}{\mathsf{Ab}}

\newcommand{\loopspacesym}{\mathsf{\Omega}}

\newcommand{\suspsym}{\mathsf{\Sigma}}
\newcommand{\susp}[1]{\suspsym #1}

\newcommand{\pto}{\to_{\sbt}}
\newcommand{\Gto}{\to_{\Grp}}
\newcommand{\Mgm}{\mathsf{Mgm}}
\newcommand{\pMgm}{{\Mgm_\sbt}}
\newcommand{\Mto}{\to_\Mgm}
\newcommand{\pMto}{\to_\pMgm}
\newcommand{\Msimeq}{\simeq_\Mgm}
\newcommand{\pMsimeq}{\simeq_\pMgm}
\newcommand{\psimeq}{\simeq_\sbt}
\newcommand{\psim}{\sim_\sbt}

\newcommand{\UU}{\mathcal{U}}

\newcommand{\refl}[1]{\ensuremath{\mathsf{refl}_{#1}}\xspace}
\newcommand{\merid}{\ensuremath{\mathsf{merid}}\xspace}
\newcommand{\transfib}[3]{\ensuremath{\mathsf{transport}^{#1}(#2,#3)\xspace}}

\DeclareMathOperator*{\colim}{\mathsf{colim}}

\newcommand{\Z}{\mathbb{Z}}
\newcommand{\N}{\mathbb{N}}

\newcommand{\stab}{\mathcal{S}}
\newcommand{\tH}{\tilde{H}}

\newcommand{\idfunc}[1]{\mathsf{id}_{#1}}
\newcommand{\ttrunc}[2]{\bigl\Vert #2\bigr\Vert_{#1}}
\newcommand{\ap}[1]{\mathsf{ap}_{#1}}

\newcommand{\true}{\mathsf{true}}

\newcommand{\sma}{\mathsf{smashing}}
\newcommand{\XY}{{X \wedge Y}}
\newcommand{\smin}{\mathsf{sm}}
\newcommand{\ptl}{\mathsf{auxl}}
\newcommand{\ptr}{\mathsf{auxr}}
\newcommand{\gluel}{\mathsf{gluel}}
\newcommand{\gluer}{\mathsf{gluer}}

\newcommand{\unitl}{\mathsf{unitl}}
\newcommand{\unitr}{\mathsf{unitr}}
\newcommand{\assoc}{\mathsf{assoc}}
\newcommand{\op}{\mathsf{op}}
\newcommand{\UUcat}{\mathsf{U}}
\newcommand{\Y}{\mathcal{Y}}
\newcommand{\sminsbt}{\smin_{\sbt}}
\newcommand{\glue}{\mathsf{glue}}
\renewcommand{\phi}{\varphi}

\newcommand{\cov}[1]{\langle #1 \rangle}

\newcommand{\lra}       {\longrightarrow}
\newcommand{\llra}[1]   {\stackrel{#1}{\lra}}  

\makeatletter
\newcommand{\ctsym}{%
  \mathchoice{\mathbin{\raisebox{0.5ex}{$\displaystyle\centerdot$}}}%
             {\mathbin{\raisebox{0.5ex}{$\centerdot$}}}%
             {\mathbin{\raisebox{0.25ex}{$\scriptstyle\,\centerdot\,$}}}%
             {\mathbin{\raisebox{0.1ex}{$\scriptscriptstyle\,\centerdot\,$}}}
  }
\newcommand{\ct}[3][]{
  \@ifnextchar\bgroup
    {#2 \mathbin{\ctsym_{#1}} \ct[#1]{#3}}
    {#2 \mathbin{\ctsym_{#1}} #3}
  }

\def\smsym{\sum}
\newcommand{\@thesum}[1]{\smsym_{(#1)}}
\newcommand{\sm}[1]{\@ifnextchar\bgroup{\@sm{#1}\sm}{\@sm{#1}}}
\newcommand{\@sm}[1]{\mathchoice{{\textstyle\@thesum{#1}}}{\@thesum{#1}}{\@thesum{#1}}{\@thesum{#1}}}

\def\prdsym{\prod}
\newcommand{\@ifnextchar@starorbrace}[2]
  {\@ifnextchar*{#1}{\@ifnextchar\bgroup{#1}{#2}}}

\newcommand{\@theprd}[1]{\prdsym_{(#1)}}
\newcommand{\@theiprd}[1]{\prdsym_{\{#1\}}}

\newcommand{\prd}{\@ifnextchar*{\@iprd}{\@prd}}
\newcommand{\@prd}[1]
  {\@ifnextchar@starorbrace
    {\@tprd{#1}\prd}
    {\@tprd{#1}}}
\newcommand{\@tprd}[1]{%
  \mathchoice{%
    {{\textstyle\@theprd{#1}}}}{\@theprd{#1}}{\@theprd{#1}}{\@theprd{#1}}}

\newcommand{\@iprd}[2]{\@ifnextchar@starorbrace%
  {\@tiprd{#2}\prd}%
  {\@tiprd{#2}}}
\newcommand{\@tiprd}[1]{
  \ifthenelse{\true}
    {\@@tiprd{#1}}
    {\@tprd{#1}}}
\newcommand{\@@tiprd}[1]{\mathchoice{{\textstyle\@theiprd{#1}}}{\@theiprd{#1}}{\@theiprd{#1}}{\@theiprd{#1}}}
\newcommand{\eqvsym}{\simeq}    
\newcommand{\eqv}[2]{
  \@ifnextchar\bgroup
    {#1 \eqvsym \eqv{#2}}
    {#1 \eqvsym #2}
  }

\def\lam#1{{\lambda}\@lamarg#1:\@endlamarg\@ifnextchar\bgroup{.\,\lam}{.\,}}
\def\@lamarg#1:#2\@endlamarg{\if\relax\detokenize{#2}\relax #1\else\@lamvar{\@lameatcolon#2},#1\@endlamvar\fi}
\def\@lamvar#1,#2\@endlamvar{(#2\,{:}\,#1)}
\def\@lameatcolon#1:{#1}

\def\lamu#1{{\lambda}\@lamuarg#1:\@endlamuarg\@ifnextchar\bgroup{.\,\lamu}{.\,}}
\def\@lamuarg#1:#2\@endlamuarg{#1}
\makeatother

\title{The Hurewicz theorem in Homotopy Type Theory}

\author{J. Daniel Christensen}
\address{University of Western Ontario, London, Ontario, Canada}
\email{jdc@uwo.ca}

\author{Luis Scoccola}
\address{University of Western Ontario, London, Ontario, Canada}
\email{lscoccol@uwo.ca}

\date{June 14, 2023}

\begin{document}

\begin{abstract}
We prove the Hurewicz theorem in homotopy type theory, i.e., that
for $X$ a pointed, $(n-1)$-connected type $(n \geq 1)$ and $A$ an abelian group,
there is a natural isomorphism
$
  \pi_n(X)^{ab} \otimes A \cong \tH_n(X; A)
$
relating the abelianization of the homotopy groups with the homology.
We also compute the connectivity of a smash product of types and
express the lowest non-trivial homotopy group as a tensor product.
Along the way, we study magmas, loop spaces, connected covers and
prespectra, and we use $1$-coherent categories to express naturality
and for the Yoneda lemma.

As homotopy type theory has models in all $\infty$-toposes, our
results can be viewed as extending known results about spaces to
all other $\infty$-toposes.
\end{abstract}


%

\maketitle

\tableofcontents

\addtocontents{toc}{\protect\setcounter{tocdepth}{1}}
\section{Introduction}\label{se:intro}

\begingroup
\renewcommand{\dfn}[1]{#1}

Homotopy type theory is a formal system which has models
in all $\infty$-toposes (\cite{KL}, \cite{LS}, \cite{S}, \cite{dB,dBB})%
\footnote{The initiality and semantics of higher inductive types still need to be fully worked out.}.
As such, it provides a convenient way to prove theorems for all $\infty$-toposes.
In addition, homotopy type theory is well-suited to being formalized
in a proof assistant (\cite{HoTTCoqRep}, \cite{doornthesis}).

Working in homotopy type theory as described in the book~\cite{hottbook},
we prove the \dfn{Hure\-wicz theorem}:

\renewcommand{\thethmintro}{H}
\begin{thmintro}[{\cref{thm:hurewicz}}]\label{thmH}
For $n \geq 1$, $X$ a pointed,
$(n-1)$-connected type, and $A$ an abelian group,
there is a natural isomorphism
\[
  \pi_n(X)^{ab} \otimes A \cong \tH_n(X; A) ,
\]
where on the left-hand-side we take the abelianization
(which only matters when $n = 1$).
In particular, when $A$ is the integers, this specializes to an
isomorphism $\pi_n(X)^{ab} \cong \tH_n(X)$.
\end{thmintro}

As mentioned above, this holds in any $\infty$-topos, and so is
more general than the well-known Hurewicz theorem in topology.
Interpreting the statement in an $\infty$-topos is somewhat subtle.
The groups that appear in the statement are internal group objects
whose underlying object is $0$-truncated (a ``set'', internally).
The quantification over $n$ means that there is a map $h : H \to \N$
in the $\infty$-topos representing a family of objects over the
natural numbers object, and that this map has a section.
In particular, since each ordinary natural number gives a global element of $\N$,
it follows that the fibre of $h$ over that element must itself have
a global element.
Continuing in this way, we deduce that for given objects $X$ and $A$
as in the statement,
the two internal group objects shown are equivalent as group objects.
For more on the interpretation of type theory,
see \cite[Section~4.2]{S15} and \cite{S} for the interpretation
in arbitrary $\infty$-topoi, and \cite{kapulkin-lumsdaine} for a more
explicit interpretation in simplicial sets.

Since we prove this theorem for an arbitrary $\infty$-topos, we must be careful
to use arguments that apply in this generality.
For example, it is not true in every $\infty$-topos that a surjective
map of sets has a section, so we cannot use the axiom of choice.
Similarly, the law of excluded middle and Whitehead's theorem can both
fail, so we must not use these results either.
Because of this, our proof is different from other known proofs.

\medskip

Before giving more details, we give some motivation for the
interest in this result, for those less familiar with traditional homotopy theory.

\subsection*{Motivation}
In topology, homotopy groups are in a certain sense the strongest invariants
of a topological space, and so their computation is an important tool
when trying to classify spaces up to homotopy.
In homotopy type theory, homotopy groups play a fundamental role in
that they capture information about iterated identity types.
Unfortunately, even in classical topology, the computation of homotopy
groups is a notoriously difficult problem.  Nevertheless, topologists
have come up with a variety of powerful tools for attacking this problem,
and one of the most basic tools is the Hurewicz theorem.
In most cases, it is much easier to compute homology groups than
homotopy groups, and so one can use the isomorphism from right to
left (with $A$ taken to be the integers) to compute certain homotopy groups.
Moreover, one can apply the theorem even when $X$ is not $(n-1)$-connected
using the following technique.
Let $X\cov{n-1}$ denote the fibre of the truncation map $X \to \ttrunc{n-1}{X}$
over the image of the basepoint.  Then $X\cov{n-1}$ is $(n-1)$-connected
and $\pi_n(X\cov{n-1}) \cong \pi_n(X)$, so $\pi_n(X)^{ab} \cong \tH_n(X\cov{n-1})$.
The Serre spectral sequence
can often be used to compute the required homology group.

%

\enlargethispage{-10pt}

\subsection*{Techniques and main results}
We first recall that for $n \geq 1$, the \dfn{$n$th homology group} $\tH_n(X; A)$
of a type $X$ with coefficients in an abelian group $A$ is defined to be
the colimit of a certain sequential diagram
\begin{equation}\label{eq:colim}
        \pi_{n+1}(X \wedge K(A,1)) \lra \pi_{n+2}(X \wedge K(A,2)) \lra \pi_{n+3}(X \wedge K(A,3)) \lra \cdots .
\end{equation}
Here $\wedge$ denotes the \dfn{smash product} and
$K(A,m)$ is the \dfn{Eilenberg--Mac Lane space} constructed in~\cite{FinsterLicata},
which is an $m$-truncated, $(m-1)$-connected, pointed type with a
canonical isomorphism $\pi_m(K(A,m)) \cong A$.

We now state one of our main results, which is used to prove the Hurewicz theorem,
and also has other consequences:

\renewcommand{\thethmintro}{S}
\begin{thmintro}[{\cref{cor:smashconnected} and \cref{thm:smashistensor}}]\label{thmS}
If $X$ is a pointed, $(n-1)$-connected type $(n \geq 1)$
and $Y$ is a pointed, $(m-1)$-connected type $(m \geq 1)$,
then $X \wedge Y$ is $(n+m-1)$-connected and $\pi_{n+m}(X \wedge Y)$ is the
tensor product of $\pi_n(X)^{ab}$ and $\pi_m(Y)^{ab}$ in a natural way.
\end{thmintro}

Taking $Y$ to be $K(A, m)$ in this result shows that the groups appearing
in the sequential diagram~\eqref{eq:colim} are tensor products of $\pi_n(X)^{ab}$ and $A$.
The proof of the Hurewicz theorem follows from showing that the induced maps
are isomorphisms, which we do in \cref{lem:stabilization}.
With this ingredient, we prove the Hurewicz theorem as \cref{thm:hurewicz}.

In order to define the isomorphism appearing in \cref{thmS},
we must give a bilinear map $\pi_n(X) \Gto \pi_m(Y) \Gto \pi_{n+m}(X \wedge Y)$.
To do so,
we define and study a more general natural map
    \[
      \sma : (X \pto Y \pto Z) \lra (\pi_n(X) \Gto \pi_m(Y) \Gto \pi_{n+m}(Z))
    \]
for any pointed types $X$, $Y$ and $Z$ and any $n, m \geq 1$.
The map we require is obtained by applying $\sma$ to
the natural map $X \pto Y \pto X \wedge Y$.

Constructing the map $\sma$ requires some work.
While it lands in group homomorphisms between ($0$-truncated) groups,
in order to construct it, we pass through \emph{magmas}.
A \dfn{magma} is a (not necessarily truncated) type $M$ with a binary operation $\cdot : M \times M \to M$,
with no conditions or coherence laws.
As a technical trick which simplifies the formalization, we work with \emph{weak magma morphisms}.
A \dfn{weak magma morphism} from a magma $M$ to a magma $N$ is a map $f : M \to N$
which \emph{merely} has the property that it respects the operations.
This is sufficient for our purposes, because when $M$ and $N$ are groups,
it reproduces the notion of group homomorphism.
All loop spaces are magmas under path concatenation, and many natural maps
involving loop spaces are weak magma morphisms.
By working with magmas, we can factor the map $\sma$ into simpler pieces,
and still land in group homomorphisms at the end, without keeping track of
higher coherences.

Proving the rest of \cref{thmS} requires a number of results that build
on work in~\cite{BuchholtzDoornRijke}.
For example, \cref{lem:conn-cover-univ,thm:groupsandspaces} are results
of~\cite{BuchholtzDoornRijke}, which we use to prove \cref{prop:connectedtotruncated}:
for $n \geq 1$, $X$ a pointed, $(n-1)$-connected type, and $Y$ a pointed,
$n$-truncated type, the map
$\loopspacesym^n : (X \pto Y) \to (\loopspacesym^n X \Mto \loopspacesym^n Y)$
is an equivalence.
In order to prove this, we prove results about connected covers in
\cref{ss:connected-covers}.

We go on to define a natural Hurewicz homomorphism
$h_n : \pi_n(X)^{ab} \otimes A \to \tH_n(X; A)$,
without assuming any connectivity hypothesis on $X$, and show that it
is unique up to a sign among such natural transformations that give
isomorphisms for $X \equiv S^n$ and $A \equiv \Z$ (\cref{thm:Hurewicz-unique}).

\subsection*{Homology}
The theory of homology in homotopy type theory is currently limited by
the absence of some important tools and facts that would make it easier to compute.
For example, we don't have complete proofs that homology satisfies
the Eilenberg--Steenrod axioms, although partial work was done by~\cite{graham}.
The Serre spectral sequence for homology has not been formalized, but
high level arguments can be found in \cite{doornthesis} and it is expected
that techniques similar to those used for cohomology will go through.
We are also missing
the fact that the homology of a cellular space can be computed cellularly (which is done for cohomology in
\cite{BuchholtzFavonia}), the universal coefficient theorem, and the relationship
between homology and localization (developed in homotopy type theory
in~\cite{CORS} and~\cite{scoccola}).

\subsection*{Structure of the paper}
\cref{se:smash-tensor} contains our work on smash products and tensor products.
After listing our conventions in \cref{ss:background},
we give the basic theory of $1$-coherent categories in \cref{ss:wildcat}.
We use this theory to express and reason about natural transformations,
and we make use of the Yoneda lemma in this setting.
In \cref{ss:connected-covers} we study connected covers.
\cref{ss:magmas} introduces magmas and weak magma morphisms, and
proves a variety of results about loop spaces,
including \cref{prop:connectedtotruncated}, mentioned above.
We also define the map $\sma$ in this section.
We introduce smash products in \cref{ss:connectivity} and prove
the connectivity part of \cref{thmS} here.
\cref{ss:abelianization} is a short section that defines abelianization
and gives a particularly efficient construction of the abelianization
of a group as a higher inductive type.
In \cref{ss:tensors}, we define tensor products of abelian groups
and prove the second part of \cref{thmS}.
\cref{ss:smashtrunc} proves results about smash products, truncation
and suspension that are needed in \cref{se:hurewicz}.

\cref{se:hurewicz} applies the results of \cref{se:smash-tensor} to homology,
leading up to the Hurewicz theorem and its consequences.
In \cref{ss:prespectra-homology}, we define prespectra and their
stable homotopy groups, and use this to define homology.
The Hurewicz theorem is proved in \cref{ss:hurewicz},
and we describe the Hurewicz homomorphism and its uniqueness up to sign
in \cref{ss:hurewicz-hom}.
In \cref{ss:applications}, we give some applications of our main results.

\subsection*{Formalization}
Formalization of these results is in progress, with help from Ali Caglayan,
using the Coq HoTT library~\cite{HoTTCoq}.
The current status can be seen at~\cite{GH}, where the \texttt{README.md} file
explains where results from the paper can be found.
Currently, we have formalized much of Section 2 but none of Section 3.
In Section 2, the only substantial result that is missing is \cref{thm:smashistensor}.
Also missing are \cref{thm:smashisequiv}
and the naturality of many of the maps defined in this section.
In our formalization, we take as axioms several results that have
been formalized in other proof assistants.

\endgroup

\addtocontents{toc}{\protect\setcounter{tocdepth}{2}}
\section{Smash products and tensor products}\label{se:smash-tensor}

In this section, we give a variety of results about loop spaces,
magmas, smash products and tensor products,
including the proof of \cref{thmS}.
None of the results in this section depend on the definition of homology,
but these results are used in the next section to prove the Hurewicz theorem.

\subsection{Background and conventions}\label{ss:background}

We follow the conventions and notation used in~\cite{hottbook}.
We assume we have a univalent universe $\UU$ closed under
higher inductive types (HITs) and contained in another universe $\UU'$.
In fact, the higher inductive types we use can all be described using pushouts and truncations.
By convention, all types live in the lower universe $\UU$, unless explicitly stated.
We implicitly use function extensionality for $\UU$ throughout.

A \define{pointed type} is a type $X$ and a choice of $x_0 : X$,
and the type of pointed types is denoted $\UU_{\sbt} \defeq \sm{X : \UU} X$.
We often keep the choice of basepoint implicit.
A \define{pointed map} between pointed types $X$ and $Y$ is a map $f : X \to Y$
and a path $p : f(x_0) = y_0$.
The type of pointed maps is denoted $X \to_{\sbt} Y \defeq \sm{f : X \to Y} f(x_0) = y_0$.

We frequently make use of functions of type $X \to Y \to Z$, and remind
the reader that this associates as $X \to (Y \to Z)$, which is the curried
form of a function $X \times Y \to Z$.

In the paper, we define the sum $m + n$ of natural numbers by induction on $n$,
so that $m+1$ is the successor of $m$.
In the HoTT library, the other convention is used, so to translate between
the paper and the formalization, one must change $m+n$ to $n+m$ everywhere.

\subsection{$1$-coherent categories}\label{ss:wildcat}

In this section, we briefly discuss the notion of $1$-coherent category,
which we use to express that various constructions are natural.
The definitions generalize those of~\cite[Section~4.3.1]{doornthesis}, which deals with
the $1$-coherent category of pointed types, except that our $\hom$ types
are unpointed.
A more general notion of wild category has been formalized in the
HoTT library~\cite{HoTTCoq} by Ali Caglayan, tslil clingman, Floris van Doorn,
Morgan Opie, Mike Shulman and Emily Riehl.

Recall that $\UU'$ is a universe such that $\UU : \UU'$.

\begin{defn}
    A \define{$1$-coherent category} $C$ consists of a type $C_0 : \UU'$, a map
    $\hom_C : C_0 \to C_0 \to \UU$, maps
    \begin{align*}
        \idfunc{} &: \prod_{a : C_0} \hom_C(a,a),\\
        - \circ - &: \prod_{a,b,c : C_0} \hom_C(b,c) \to \hom_C(a,b) \to \hom_C(a,c),
    \end{align*}
    and equalities
    \begin{align*}
        \unitl &: \prod_{a, b : C_0} \prod_{f : \hom_C(a,b)} \idfunc{b} \circ f = f\\
        \unitr &: \prod_{a, b : C_0} \prod_{f : \hom_C(a,b)} f \circ \idfunc{a} = f\\
        \assoc &: \prod_{a,b,c,d : C_0} \prod_{f : \hom_C(a,b)} \prod_{g : \hom_C(b,c)} \prod_{h : \hom_C(c,d)}
            (h \circ g) \circ f = h \circ (g \circ f),
    \end{align*}
    witnessing left and right unitality and associativity, respectively.
    We do not assume coherence laws or that any of the types are truncated.

    If $C$ is a $1$-coherent category,
    the elements of $C_0$ are called \define{objects} and, for objects $a,b : C_0$,
    the elements of $\hom_C(a,b)$ are called \define{morphisms} from $a$ to $b$.
\end{defn}

The wild $1$-categories considered in~\cite{HoTTCoq} allow $2$-cells to be specified,
which are then used in place of the identity types in the above equalities.
For simplicity, we use the identity types.

\begin{eg}
    There is a $1$-coherent category $\UUcat$ of types, with $\UUcat_0 \defeq \UU$
    and $\hom_\UUcat(X,Y) \defeq X \to Y$ for every pair of types $X,Y : \UU$.
    Identity morphisms, composition, unitalities, and associativity all work in the expected way.
\end{eg}

\begin{eg}
    There is a $1$-coherent category $\Grp$ of groups whose objects are the set-level groups, that is,
    $0$-truncated types equipped with an associative binary operation, a unit and inverses.
    The morphisms are standard group homomorphisms.
    
    Similarly, there is $1$-coherent category $\Ab$ of abelian groups.
\end{eg}

\begin{eg}
    Any precategory in the sense of \cite[Definition~9.1.1]{hottbook} gives rise to a $1$-coherent category,
    simply by forgetting that its hom types are sets.
    Moreover, the notions of isomorphism, functor, and natural transformation given in \cite[Section~9]{hottbook}
    are equivalent to the notions we give in this section, in the case of precategories.
\end{eg}

Many constructions one can carry out with categories are easy to extend to $1$-coherent categories.
We mention two that are particularly important for us.
Given a $1$-coherent category $C$, we can form the \define{opposite} $1$-coherent category
$C^\op$ by letting the type of objects of $C^\op$ be $C_0$, and $\hom_{C^\op}(a,b) \defeq \hom_C(b,a)$
for all $a,b : C_0$. The rest of the structure is straightforward to define.

Given $1$-coherent categories $C$ and $D$ one can form a \define{product} $1$-coherent category, denoted by $C \times D$.
The underlying type of $C \times D$ is simply $C_0 \times D_0$, and
$\hom_{C \times D}((c,d),(c',d')) \defeq \hom_C(c,c') \times \hom_D(d,d')$.
The rest of the structure is again straightforward to define.

\begin{defn}
    Let $C$ be a $1$-coherent category, $a,b : C_0$, and $f : \hom_C(a,b)$.
    An \define{isomorphism structure} for $f$ is given by morphisms $g,h : \hom_C(b,a)$
    together with paths $l : g \circ f = \idfunc{a}$ and $r : f \circ h = \idfunc{b}$.
\end{defn}

In many cases, such as in the $1$-coherent category $\UUcat$, being an isomorphism is a mere property
of a morphism.
The wild $1$-categories considered in~\cite{HoTTCoq} allow biinvertibility
to be replaced by more general notions of isomorphism.

\begin{defn}
    A \define{$1$-coherent functor} $F$ between $1$-coherent categories $C$ and $D$,
    usually denoted by $F : C \to D$, consists of a map $F_0 : C_0 \to D_0$, a map
    \[
        F_1 : \prod_{a,b : C_0} \hom_C(a,b) \to \hom_D(F_0(a),F_0(b)),
    \]
    and equalities
    \begin{align*}
        F_{\idfunc{}} &: \prod_{a : C_0} F(\idfunc{a}) = \idfunc{F(a)},\\
        F_\circ &: \prod_{a,b,c : C_0} \prod_{f : \hom_C(a,b)} \prod_{g : \hom_C(b,c)} F_1(g) \circ F_1(f) = F(g\circ f),
    \end{align*}
    witnessing the functoriality of $F$.
\end{defn}

\begin{eg}
    For a $1$-coherent category $C$ and an object $a : C_0$,
    we can define a $1$-coherent corepresentable functor $\Y^a : C \to \UUcat$.
    On objects we have $\Y^a_0(b) \defeq \hom_C(a,b)$.
    The action on morphisms is defined as $\Y^a_1(f) :\equiv \lambda g. f \circ g : \hom_C(a,b) \to \hom_C(a,c)$ for $f : \hom_C(b,c)$.
    The witnesses of functoriality, that is $\Y^a_{\idfunc{}}$ and $\Y^a_\circ$, are defined using the equalities $\unitl$ and $\assoc$ of $C$, respectively.
\end{eg}

\begin{defn}
    Let $C$ and $D$ be $1$-coherent categories and let $F,G : C \to D$ be $1$-coherent functors.
    A \define{$1$-coherent natural transformation} $\alpha$ from $F$ to $G$, usually denoted by $\alpha : F \to G$,
    consists of a map
    \[
        \alpha_0 : \prod_{a : C_0} \hom_D(F(a),G(a)),
    \]
    and equalities
    \[
        \alpha_1 : \prod_{a, b : C_0} \prod_{f : \hom_C(a,b)} \alpha_0(b) \circ F_1(f) = G_1(f) \circ \alpha_0(a).
    \]
\end{defn}

\begin{defn}
    Let $\alpha : F \to G$ be a $1$-coherent natural transformation between $1$-coherent
    functors $F, G : C \to D$, for $C$ and $D$ $1$-coherent categories.
    An \define{isomorphism structure} for $\alpha$ is given by an isomorphism structure
    for each of its components.
    A \define{natural isomorphism} is given by a natural transformation together with an isomorphism structure.
\end{defn}

The following lemma is straightforward.

\begin{lem}
    Let $C$ and $D$ be $1$-coherent categories, $F,G,H : C \to D$ $1$-coherent functors, and
    $\alpha : F \to G$ and $\beta : G \to H$ $1$-coherent natural transformations.
    Then, by defining $(\beta \circ \alpha)(c) \defeq \beta(c) \circ \alpha(c)$ and the naturality
    squares by composing the naturality squares of $\alpha$ and $\beta$, one obtains
    a natural transformation $\beta \circ \alpha : F \to H$. Moreover, if both $\alpha$ and $\beta$
    are natural isomorphisms, so is $\beta \circ \alpha$.\qed
\end{lem}

The following is a $1$-coherent version of the fact that the Yoneda functor is an embedding.

\begin{prop}[{\cite{HoTTCoq}}]\label{1-coherent-yoneda-embedding}
    Let $C$ be a $1$-coherent category and let $a,b : C_0$.
    Assume given a $1$-coherent natural isomorphism $\alpha : \Y^b \to \Y^a$.
    Then $i \defeq \alpha_0(b)(\idfunc{b}) : \hom_C(a,b)$ is part of an isomorphism between $a$ and $b$,
    and it satisfies, for every $c : C_0$,
    \[
        \alpha_0(c) = \lambda g. g \circ i
    \]
    as maps $\hom_C(b,c) \to \hom_C(a,c)$.  \qed
\end{prop}

The proof is the same as the usual proof, and has been formalized in
the HoTT library~\cite{HoTTCoq}.
Note that we are not claiming that the naturality proofs for $\alpha$
can be recovered using the associativity of composition.

\subsection{Connected covers}\label{ss:connected-covers}

In order to generalize a result of Buchholtz, van Doorn and Rijke
(see \cref{thm:groupsandspaces}) to the case where $Y$
has no connectivity assumption, we prove some results about connected covers.
In this section, we fix $n \geq -1$.

\begin{defn}
    A type $X$ is \define{$n$-connected} if $\ttrunc{n}{X}$ is contractible.
\end{defn}

For $X$ pointed, it is equivalent to require that $\pi_i(X)$ be trivial
for all $i \leq n$.
Every pointed type is $(-1)$-connected.

\begin{defn}
    Let $X$ be a pointed type.
    The \define{$n$-connected cover}
    $X\cov{n}$ of $X$ is defined to be the
    fibre of the pointed map $X \pto \ttrunc{n}{X}$.
\end{defn}

Note that $X\cov{n}$ is indeed $n$-connected and that
we have a canonical pointed map $i : X\cov{n} \pto X$
which induces an equivalence on the homotopy groups $\pi_k$ for $k > n$.
In fact, this map has a stronger universal property:

\begin{defn}
    A pointed map $f : X \pto Y$ is an \define{$\cov{n}$-equivalence} if
    for any pointed, $n$-connected type $Z$, post-composition by $f$ gives an equivalence
    \[
        (Z \pto X) \llra{\sim} (Z \pto Y).
    \]
\end{defn}

\begin{lem}[{\cite[Lemma~6.2]{BuchholtzDoornRijke}}]\label{lem:conn-cover-univ}
    Let $X$ be a pointed type.
    Then the map $i : X\cov{n} \pto X$ is an $\cov{n}$-equivalence.  \qed
\end{lem}

It follows that the operation sending $X$ to $X\cov{n}$ is functorial
in a unique way making $i : X\cov{n} \pto X$ natural, and that
a map $f$ is an $\cov{n}$-equivalence if and only if $f\cov{n}$ is an
equivalence.

Note that there is a $1$-coherent category with objects all pointed types
and morphisms given by pointed functions. We denote this $1$-coherent category
by $\UUcat_\sbt$. There are $1$-coherent functors $\suspsym, \loopspacesym : \UUcat_\sbt \to \UUcat_\sbt$
forming a $1$-coherent adjunction, in the following sense.

\begin{lem}[{\cite[Lemma~6.5.4]{hottbook}}]\label{susp-loops-adjunction}
    Let $X$ and $Y$ be pointed types.
    There is an equivalence
    \[
        (\suspsym X \to_\sbt Y) \simeq (X \to_\sbt \loopspacesym Y),
    \]
    natural in $X$ and $Y$.
    Here, we are interpreting $(\suspsym(-) \to_\sbt -)$ and $(- \to_\sbt \loopspacesym(-))$ 
    as $1$-coherent functors $\UUcat_\sbt^\op \times \UUcat_\sbt \to \UUcat$.  \qed
\end{lem}

The naturality is not proven in~\cite{hottbook}, but is proven in
the HoTT library~\cite{HoTTCoq}.

The following two facts will be used in \cref{prop:connectedtotruncated}.

\begin{prop}\label{loops-conn-equiv}
    Let $f : X \pto Y$ be a pointed map.
    If $f$ is an $\cov{n+1}$-equivalence, then $\loopspacesym f$ is an
    $\cov{n}$-equivalence.
\end{prop}

\begin{proof}
    Let $A$ be an $n$-connected, pointed type.
    By naturality of the adjunction between suspension and loops (\cref{susp-loops-adjunction}),
    we have a commutative square
\[
\begin{tikzcd}
  (\suspsym A \pto X) \ar[d,"\sim",swap] \ar[r,"f \circ -"] & (\suspsym A \pto Y) \ar[d,"\sim"] \\
  (A \pto \loopspacesym X) \ar[r,"\loopspacesym f \circ -"] & (A \pto \loopspacesym Y)
\end{tikzcd}
\]
    in which the vertical maps are equivalences.
    Since the suspension of an $n$-connected type is $(n+1)$-connected,
    the top map is also an equivalence.
    Therefore, the bottom map is an equivalence, as required.
\end{proof}

\begin{prop}\label{-1-equivalence}
    Let $f : X \pto Y$ be a pointed map.  If $f$ is a $\cov{-1}$-equivalence,
    then $f$ is an equivalence.
\end{prop}

\begin{proof}
    Since $S^0$ is $(-1)$-connected, we know that $f$ induces an
    equivalence $(S^0 \pto X) \pto (S^0 \pto Y)$.
    Moreover, $(S^0 \pto Z)$ is equivalent to $Z$ for any pointed type $Z$,
    and this equivalence is natural.
    It follows that $f$ is an equivalence.
\end{proof}

This also follows from the facts that $Z\cov{-1} \pto Z$ is an equivalence
for any pointed $Z$, and that $f\cov{-1}$ is an equivalence.

\subsection{Loop spaces and magmas}\label{ss:magmas}
In this section, we study loop spaces and the natural magma structures that they carry
and define the map $\sma$ that plays an important role in this paper.
We begin by generalizing the following result of Buchholtz, van Doorn and Rijke.

\begin{thm}[{\cite[Theorem~5.1]{BuchholtzDoornRijke}}]\label{thm:groupsandspaces}
    Let $n \geq 1$.
    For $X$ and $Y$ pointed, $(n-1)$-connected, $n$-truncated types,
    the map
    \[
        \loopspacesym^n : (X \pto Y) \lra_{\sbt} (\loopspacesym^n X \Gto \loopspacesym^n Y).
    \]
    is an equivalence.
    \qed
\end{thm}

In order to state our generalization, we introduce the notion of magma.

\begin{defn}
    A \define{magma} is given by a type $X$ together with an operation $\cdot_X : X \times X \to X$.
    A \define{map of magmas} is given by a map $f : X \to Y$ between the underlying types
    that merely respects the operations.
    More formally, we define
\[
    X \Mto Y := \sm{f : X \to Y} \ttrunc{-1}{\prd{x, x' : X} f(x \cdot_X x') = f(x) \cdot_Y f(x')} .
\]
    An \define{equivalence of magmas} is a map of magmas whose underlying map
    is an equivalence.  We write $X \Msimeq Y$ for the type of magma equivalences
    from $X$ to $Y$.
    Magmas form a $1$-coherent category that we denote $\Mgm$.
    We will omit the subscript on the operation $\cdot$ when it is clear from context.
\end{defn}

The propositional truncation in the definition of magma map is a technical
trick to simplify the formalization.
With our definition, the type of equalities between magma maps
is equivalent to the type of equalities between the underlying maps.
All of our results should go through without the truncation,
but omitting it leads to path algebra that is not needed
in order to get our later results.
The maps we are considering should be called ``weak magma maps,'' but since
they are the only maps we use, we simply call them ``magma maps'' in this paper.

\begin{defn}
  A \define{pointed magma} is a magma $X$ with a chosen point $x_0 : X$
  and an equality $x_0 \cdot x_0 = x_0$.
  A \define{map of pointed magmas} is a pointed map $f : X \pto Y$ whose
  underlying map $f : X \to Y$ is a map of magmas.
  We write $X \pMto Y$ for the type of pointed magma maps.
  An \define{equivalence of pointed magmas} is a map of pointed magmas
  whose underlying map is an equivalence.
  We write $X \pMsimeq Y$ for the type of pointed magma equivalences.
  Pointed magmas form a $1$-coherent category that we denote $\pMgm$.
\end{defn}

There are no propositional truncations in the above definition, except
for the one in the definition of magma map.

\begin{rmk}\label{rmk:magmas}
The loop space $\loopspacesym X$ is a pointed magma for any pointed type $X$,
with path concatenation as the operation, reflexivity as the basepoint,
and a higher reflexivity as the proof that the basepoint is idempotent.
There
is a natural map\break $\loopspacesym : (X \pto Y) \pto (\loopspacesym X \pMto \loopspacesym Y)$,
which can be iterated.
Any magma map $\loopspacesym X \Mto \loopspacesym Y$ induces a group homomorphism $\pi_1(X) \Gto \pi_1(Y)$.
Also note that for groups $G$ and $H$, $(G \Gto H) \simeq (G \Mto H)$,
where we write $G \Gto H$ for the type of group homomorphisms.
(We assume that all groups have an underlying type that is a set,
which means that the propositional truncation can be removed.)

When $X$ is a pointed magma and $G$ is a group, every magma map $X \Mto G$
can be made pointed in a unique way, so the forgetful map $(X \pMto G) \to (X \Mto G)$
is an equivalence.

When $A$ is a pointed type and $X$ is a pointed magma, the type $A \pto X$
of pointed maps is a pointed magma under the pointwise operation.
The requirement that the basepoint $x_0 : X$ be idempotent ensures that for
$f, g : A \pto X$, $f \cdot g$ is again pointed:
$(f \cdot g)(a_0) \equiv f(a_0) \cdot g(a_0) = x_0 \cdot x_0 = x_0$.

Similarly, when $Y$ is a pointed magma and $Z$ is a pointed type,
the type $Y \pMto \loopspacesym^2 Z$ of pointed magma maps
and the type $Y \Mto \loopspacesym^2 Z$ of all magma maps
are pointed magmas under the pointwise operation.
This uses that path composition in the double loop space is commutative
(by Eckmann--Hilton) and associative.
(More precisely, we only use that the operation is merely commutative
and merely associative, which will be convenient in \cref{defn:sm}.)
\end{rmk}

With this background, we can now state our first generalization of \cref{thm:groupsandspaces}.

\begin{prop}\label{prop:connectedtotruncated}
    Let $n \geq 1$,
    let $X$ be a pointed, $(n-1)$-connected type, and let $Y$ be a pointed, $n$-truncated type.
    Then the map
    \[
        \loopspacesym^n : (X \pto Y) \llra{\sim} (\loopspacesym^n X \Mto \loopspacesym^n Y),
    \]
    is an equivalence, natural in $X$ and $Y$.
    Similarly,
    \[
        \loopspacesym^n : (X \pto Y) \llra{\sim} (\loopspacesym^n X \pMto \loopspacesym^n Y),
    \]
    is a natural equivalence.
\end{prop}

\begin{proof}
    Since $\Omega^n Y$ is a group, the second equivalence follows from the first,
    using \cref{rmk:magmas}, so we focus on the first one.
    By the functoriality of $\loopspacesym^n$, the diagram
\[
        \begin{tikzpicture}
          \matrix (m) [matrix of math nodes,row sep=2em,column sep=3em,minimum width=2em]
            { \left(X \pto Y\right) & \left(\loopspacesym^n X \Mto \loopspacesym^n Y\right) \\
            \left(\ttrunc{n}{X} \pto Y\right) & \left(\loopspacesym^n\left(\ttrunc{n}{X}\right) \Mto \loopspacesym^n Y \right)\\
            \left(\ttrunc{n}{X} \pto Y\cov{n-1}\right) & \left(\loopspacesym^n\left(\ttrunc{n}{X}\right) \Mto \loopspacesym^n\left(Y\cov{n-1}\right)\right) \\};
          \path[->]
            (m-1-1) edge node [above] {} (m-1-2)
            (m-2-1) edge node [above] {} (m-2-2)
                    edge node [left] {} (m-1-1)
            (m-3-1) edge node [above] {} (m-3-2)
                    edge node [left] {} (m-2-1)
            (m-2-2) edge node [right] {} (m-1-2)
            (m-3-2) edge node [right] {} (m-2-2)
            ;
        \end{tikzpicture}
\]
    commutes, where the vertical maps are induced by the maps
    $|-|_n : X \pto \ttrunc{n}{X}$ and $i : Y\cov{n-1} \pto Y$.
    The vertical maps on the left are equivalences by the universal properties
    of truncations and of connected covers.

    To see that the upper vertical map on the right is an equivalence, let
    $f$ denote the map
    $\loopspacesym^n(|-|_n) : \loopspacesym^n X \pto \loopspacesym^n(\ttrunc{n}{X})$.
    This map is $0$-connected, since $|-|_n$ is $n$-connected and $\loopspacesym$
    decreases connectivity.
    Since $\loopspacesym^n Y$ is a set, it follows that $f$ 
    induces an equivalence 
    $(\loopspacesym^n(\ttrunc{n}{X}) \to \loopspacesym^nY) \to
     (\loopspacesym^n(X) \to \loopspacesym^nY)$.
    Given $g : \loopspacesym^n(\ttrunc{n}{X}) \to \loopspacesym^nY$,
    we need to show that $g$ merely preserves the magma structures
    if and only if $g \circ f$ merely preserves the magma structures.
    The map $f$ induces an equivalence
\[
  \Bigg(\prod_{a, b : \loopspacesym^n(\Vert X \Vert_n)} g(a \cdot b) = g(a) \cdot g(b)\Bigg)
  \simeq
  \Bigg(\prod_{a, b : \loopspacesym^n(X)} g(f(a) \cdot f(b)) = g(f(a)) \cdot g(f(b))\Bigg),
\]
    since $f$ is $0$-connected and the identity types are sets (in fact, propositions).
    Note that $f$, being defined using the functoriality of $\loopspacesym$,
    preserves the concatenation operation (without any propositional truncation).
    It follows that the type on the right is
    equivalent to the type of proofs that $g \circ f$ preserves the magma structure.
    Therefore, the propositional truncations are also equivalent,
    so $f$ induces an equivalence on the types of magma maps.

    The lower vertical map on the right is an equivalence since
    $\loopspacesym^n(i)$ is an equivalence of magmas: it is certainly a map of magmas, and
    the fact that it is an equivalence follows from
    \cref{loops-conn-equiv,-1-equivalence}.
    
    The bottom horizontal map is an equivalence by \cref{thm:groupsandspaces},
    and so the top horizontal map is an equivalence, as required.

    The fact that $\loopspacesym^n$ is natural in $X$ and $Y$ follows from
    the functoriality of $\loopspacesym^n$ as an operation from pointed maps to
    magma maps, which is straightforward to check.
\end{proof}

Our next goal is to define the map $\sma$, using the following lemmas.

\begin{lem}\label{lem:homotopygroupmappingspace}
    Let $n \geq 1$ and let $Y$ and $Z$ be pointed types.
    Then there is an equivalence of pointed magmas
    \[
        \loopspacesym^n(Y \pto Z) \pMsimeq \left(Y \pto \loopspacesym^n Z \right),
    \]
    natural in $Y$ and $Z$.
    Here we are regarding $\loopspacesym^n(- \pto -)$ and $- \pto \loopspacesym^n(-)$
    as $1$-coherent functors $\UUcat_\sbt \times \UUcat_\sbt \to \pMgm$.
\end{lem}

On the right-hand-side, we are using the pointwise magma structure
described in \cref{rmk:magmas}.

\begin{proof}
    We prove this for $n = 1$, and then iterate, using that the functor
    $\loopspacesym$ sends pointed equivalences to equivalences of pointed magmas.

    In order to prove that our equivalence respects the magma structures,
    it is best to generalize:  for $f, g : Y \pto Z$ we define an equivalence
\[
    \phi : (f = g) \llra{\sim} \sum_{K : f \sim g} \, K(y_0) = f_0 \cdot g_0^{-1}.
\]
    Here $K$ is a homotopy, $y_0$ is the basepoint of $Y$,
    and $f_0$ and $g_0$ are the paths witnessing that $f$ and $g$ are pointed.
    This equivalence is a variant of the standard result that equalities of
    pointed maps are equivalent to pointed homotopies; the particular choice
    of the right-hand-side means that when $f$ and $g$ are the constant map
    $Y \pto Z$ pointed by $\refl{}$, we obtain a pointed equivalence
    \[
        \loopspacesym(Y \pto Z) \psimeq \left(Y \pto \loopspacesym Z \right).
    \]
    Our pointed homotopies can be composed, and we show that $\phi$
    sends composition of paths to composition of homotopies by first doing
    induction on the paths to reduce the goal to
\[
    \phi(\refl{}) = \phi(\refl{}) \cdot \phi(\refl{})
\]
    and then using path induction to assume that $f_0$ is $\refl{}$.
    We conclude that
    \[
        \loopspacesym(Y \pto Z) \pMsimeq \left(Y \pto \loopspacesym Z \right).
    \]

    To prove naturality in $Y$, consider a pointed map $h : Y \pto Y'$.
    We must show that the following square commutes:
    \[
    \begin{tikzcd}[column sep=3em]
      \loopspacesym(Y' \pto Z) \ar[d,"\phi"'] \ar[r,"\loopspacesym(- \circ h)"] &
      \loopspacesym(Y  \pto Z) \ar[d,"\phi"] \\
      (Y' \pto \loopspacesym Z) \ar[r,"- \circ h"'] &
      (Y  \pto \loopspacesym Z) .
    \end{tikzcd}
    \]
    By path induction, we can assume that $h$ is strictly pointed, i.e.,
    that the given path $h_0 : h(y_0) = y_0'$ is reflexivity.
    In this case, writing $c : Y \pto Z$ and $c' : Y' \pto Z$ for the constant maps,
    we have that $c' \circ h$ and $c$ are definitionally equal as pointed maps.
    Therefore, the corners and vertical maps in the required square are definitionally
    equal to those in the square
    \[
    \begin{tikzcd}[column sep=3em]
      c' = c' \ar[d,"\phi"'] \ar[r,"\ap{- \circ h}"] &
      c' \circ h = c' \circ h \ar[d,"\phi"] \\
      c' \psim c' \ar[r,"\mathsf{wh}_h"] &
      c' \circ h \psim c' \circ h  ,
    \end{tikzcd}
    \]
    where $\psim$ denotes the type of pointed homotopies defined above, and
    $\mathsf{wh}_h$ denotes prewhiskering with $h$.
    One can check that the horizontal arrows are homotopic to those
    in the required square, so it remains to show that the new square
    commutes.
    To show this, one generalizes from $c' = c'$ to $f = g$, in which
    case the commutativity follows by path induction.

    The proof of naturality in $Z$ is very similar.
    Since both naturalities have been formalized, we give no further details.
\end{proof}

\begin{lem}\label{lem:magmamap_loops_functor}
\def\lnZ{\loopspacesym^n Z}
  Let $n, m \geq 1$ and let $Y$ and $Z$ be pointed types.
  The action of $\loopspacesym^m$ on maps gives a pointed magma map
  \[
    (Y \pto \lnZ) \lra_{\pMgm} (\loopspacesym^m Y \pMto \loopspacesym^m \lnZ).
  \]
  Moreover, the forgetful maps
  \[
    (\loopspacesym^m Y \pMto \loopspacesym^m \lnZ) \lra_{\pMgm}
    (\loopspacesym^m Y \Mto \loopspacesym^m \lnZ)
  \]
  and
  \[
    (\loopspacesym^m Y \pMto \loopspacesym^m \lnZ) \lra_{\pMgm}
    (\loopspacesym^m Y \pto \loopspacesym^m \lnZ) \hspace*{1.2em} 
  \]
  are also pointed magma maps.
  In all cases, we are using the pointwise magma structure described in \cref{rmk:magmas}.
  These maps are all natural.
\end{lem}

\begin{proof}
\def\l1Z{\loopspacesym Z}
That the forgetful maps are natural pointed magma maps is straightforward,
so we focus on the first map.
By replacing $Z$ with $\loopspacesym^{n-1} Z$, we can assume that $n = 1$.
To prove that $\Omega^m$ is a natural pointed magma map, we induct on $m$.
For the inductive step, we define $\Omega^{m+1}$ to be the composite
\[
\begin{aligned}
  (Y \pto \l1Z) &\llra{\loopspacesym^m} (\loopspacesym^m Y \pMto \loopspacesym^m \l1Z) \\
  &\lra (\loopspacesym^m Y \pto \loopspacesym^m \l1Z) \\
  &\llra{\loopspacesym} (\loopspacesym^{m+1} Y \pMto \loopspacesym^{m+1} \l1Z)
\end{aligned}
\]
so that the claim follows from the inductive hypothesis, the fact that
the middle forgetful map is a natural pointed magma map, and the $m=1$ case.

It remains to prove the $m=1$ case.
It is easy to see that for $f : Y \pto \l1Z$, $\loopspacesym f$ is a
pointed magma map.
Next we must show that given $f, g : Y \pto \l1Z$,
$\loopspacesym (f \cdot g)$ and $(\loopspacesym f) \cdot (\loopspacesym g)$
are equal as pointed magma maps, where $\cdot$ denotes the pointwise operations.
Because we are using weak magma maps, it is equivalent to show that these
two maps are equal as pointed maps, or in other words that there is a
pointed homotopy
$\loopspacesym (f \cdot g) \psim (\loopspacesym f) \cdot (\loopspacesym g)$.
The underlying homotopy involves some path algebra, and ultimately follows
from the fact that horizontal and vertical composition agree in the codomain,
which is a double-loop space.
The pointedness of the homotopy follows by a simple path induction on
the paths $f(y_0) = \refl{}$ and $g(y_0) = \refl{}$, after generalizing
$f(y_0)$ and $g(y_0)$ to arbitrary loops.
The argument in this paragraph has been formalized.

The naturality of $\Omega$ follows from the fact that for pointed
maps $h$ and $k$, $\Omega(h \circ k) = \Omega(h) \circ \Omega(k)$
as pointed maps, where again we are taking advantage of the fact
that we are using weak magma maps.
\end{proof}

\begin{defn}\label{defn:sm}
    For pointed types $X$, $Y$, and $Z$ and natural numbers $n,m \geq 1$, we have maps:
    \begin{equation}\label{eq:sm}
    \begin{aligned}
        (X \pto Y \pto Z) &\lra (\loopspacesym^n X \Mto \loopspacesym^n (Y \pto Z))\\
                                  &\llra{\sim} (\loopspacesym^n X \Mto (Y \pto \loopspacesym^n Z))\\
                                  &\lra (\loopspacesym^n X \Mto (\loopspacesym^m Y \pMto \loopspacesym^m \loopspacesym^n Z))\\
                                  &\llra{\sim} (\loopspacesym^n X \Mto (\loopspacesym^m Y \pMto \loopspacesym^{n+m} Z))\\
                                  &\lra (\pi_n(X) \Gto \pi_m(Y) \Gto \pi_{n+m}(Z)).
    \end{aligned}
    \end{equation}
    These maps are natural in $X$, $Y$, and $Z$.
    The first and third arrows apply $\loopspacesym^n$ and $\loopspacesym^m$
    to morphisms, using \cref{lem:magmamap_loops_functor}.
    The second arrow is an equivalence by \cref{lem:homotopygroupmappingspace}.
    To understand the fourth arrow, write $m = k + 1$ for some $k : \N$.
    Then $\loopspacesym^k \loopspacesym^n Z = \loopspacesym^{n+k} Z$ as pointed types.
    Applying $\loopspacesym$ on the outside, we see that
    $\loopspacesym^m \loopspacesym^n Z = \loopspacesym^{n+m} Z$ as magmas.
    Since the magma structure on the set of magma maps only uses that
    the iterated loop space is \emph{merely} commutative and \emph{merely} associative,
    we can conclude that
    $(\loopspacesym^m Y \pMto \loopspacesym^m \loopspacesym^n Z) =
     (\loopspacesym^m Y \pMto \loopspacesym^{n+m} Z)$
    as magmas.
    From this we deduce the required equivalence.
    The fifth arrow applies $0$-truncation on the inside and then on the outside.
    Let
    \[
      \sma : (X \pto Y \pto Z) \lra (\pi_n(X) \Gto \pi_m(Y) \Gto \pi_{n+m}(Z))
    \]
    denote the composite.
\end{defn}

The map $\sma$ corresponds to the following construction in topology,
which uses the smash product from the next section.
Given a map $f : X \pto Y \pto Z$ and homotopy classes $\alpha : \pi_n(X)$
and $\beta : \pi_m(Y)$, one can smash representatives of the homotopy classes
together to get an element $\alpha \wedge \beta : \pi_{n+m}(X \wedge Y)$.
The adjoint $X \wedge Y \pto Z$ of $f$ then induces a map taking this
to an element of $\pi_{n+m}(Z)$ which (up to sign) is $\sma(f, \alpha, \beta)$.
This correspondence motivates the name.

Since we'll use it several times, we quote the following result from~\cite{BuchholtzDoornRijke}.

\begin{lem}[{\cite[Corollary~4.3]{BuchholtzDoornRijke}}]\label{lem:mappingspacetrunc}
    Let $m \geq 0$ and $n \geq -1$. If $Y$ is a pointed, $(m-1)$-connected type and $Z$ is a pointed, $(n+m)$-truncated type,
    then the type $Y \pto Z$ is $n$-truncated.\qed
\end{lem}

The last result in this section plays an important role in our proof,
and can be thought of as a generalization of \cref{thm:groupsandspaces}
to functions with two arguments.

\begin{thm}\label{thm:smashisequiv}
    Let $n, m \geq 1$.
    If $X$ is a pointed $(n-1)$-connected type, $Y$ is a pointed $(m-1)$-connected type,
    and $Z$ is a pointed $(n+m)$-truncated type, then the map $\sma$ is an equivalence.
\end{thm}

\begin{proof}
    The first arrow in \cref{eq:sm} is an equivalence by
    \cref{lem:mappingspacetrunc} and \cref{prop:connectedtotruncated}.
    The third arrow is an equivalence by \cref{prop:connectedtotruncated}.
    To show that the fifth arrow is an equivalence, one uses the same methods
    as in the proof of \cref{prop:connectedtotruncated}, using that
    $\loopspacesym^{n+m} Z$ is a set.
\end{proof}

\subsection{The connectivity of smash products}\label{ss:connectivity}

We recall some basic facts about smash products, and then prove a
result about their connectivity.

\begin{defn}
    For pointed types $X$ and $Y$, the \define{smash product} $X \wedge Y$
    is defined to be the higher inductive type with constructors:
    \begin{itemize}
    \item $\smin : X \times Y \to \XY$.
    \item $\ptl : \XY$.
    \item $\ptr : \XY$.
    \item $\gluel : \prd{y:Y} \smin(x_0,y) = \ptl$.
    \item $\gluer : \prd{x:X} \smin(x,y_0) = \ptr$.
    \end{itemize}
    The smash product is pointed by $\smin(x_0, y_0)$.
    It has the expected induction principle.
\end{defn}

It is straightforward to see that the smash product is a functor.
That is, given pointed maps $f : X \pto X'$ and $g : Y \pto Y'$ between pointed types,
there is a pointed map $f \wedge g : X \wedge Y \pto X' \wedge Y'$
defined by induction on the smash product in the evident way,
and this operation respects identity maps and composition.

Given pointed types $X$ and $Y$, the constructors of the smash product $X\wedge Y$
combine to give a map $X \to_\sbt (Y \to_\sbt X \wedge Y)$, which we now describe.

\begin{defn}\label{defn:sminsbt}
    Let $X,Y : \UU_\sbt$. Currying the constructor $\smin$, we get a map
    $X \to (Y \to X \wedge Y)$. Using the constructor $\gluer$ twice, this map
    lifts to a map $X \to (Y \to_\sbt X \wedge Y)$. Similarly, using $\gluel$,
    this last map lifts to a map $\sminsbt : X \to_\sbt (Y \to_\sbt X \wedge Y)$.
\end{defn}

The following adjunction between pointed maps and smash products is fundamental to our work.

\begin{lem}[{\cite[Theorem~4.3.28]{doornthesis}}]\label{lem:smashhomadjunction}
    Let $X$, $Y$, and $Z$ be pointed types.
    The map
    \[
        (X \wedge Y \to_\sbt Z) \lra_\sbt (X \to_\sbt (Y \to_\sbt Z)),
    \]
    induced by precomposition with $\smin_{\sbt}$ is a pointed equivalence,
    natural in $X$, $Y$, and $Z$.
    Here, we are interpreting $(- \wedge - \to_\sbt -)$ and $(- \to_\sbt (- \to_\sbt -))$
    as $1$-coherent functors $\UUcat_\sbt^\op \times \UUcat_\sbt^\op \times \UUcat_\sbt \to \UUcat_\sbt$.
    \qed
\end{lem}

Note that, by construction, $\sminsbt : X \to_{\sbt} Y \to_{\sbt} X \wedge Y$ is the adjunct of the
identity map $X \wedge Y \to_{\sbt} X \wedge Y$.

In the form stated here, \cref{lem:smashhomadjunction} has been formalized by~\cite{doornthesis}.
A stronger statement, which roughly involves regarding the category $\UUcat_\sbt$
as being \emph{enriched} over $\UUcat_\sbt$, has not yet been proven,
but we do not use this stronger form.

We now give a bound on the connectivity of smash products,
proving the first part of \cref{thmS} from the Introduction.

\begin{cor}\label{cor:smashconnected}
    Let $n,m \geq 0$, let $X$ be a pointed, $(n-1)$-connected type, and let $Y$ be a pointed, $(m-1)$-connected type.
    Then $X \wedge Y$ is $(n+m-1)$-connected.
\end{cor}

\begin{proof}
    It is enough to show that the truncation map $X \wedge Y \to \ttrunc{n+m-1}{X \wedge Y}$ is nullhomotopic.
    Since the truncation map is pointed, this follows from the following
    more general fact: for any pointed, $(n+m-1)$-truncated type $Z$, the type $X \wedge Y \pto Z$ is contractible.
    Indeed, by \cref{lem:smashhomadjunction},
    we have $(X \wedge Y \pto Z) \simeq (X \pto Y \pto Z)$.
    By \cref{lem:mappingspacetrunc}, the type $Y \pto Z$ is $(n-1)$-truncated.
    Therefore, using \cref{lem:mappingspacetrunc} again, we see that the type $X \pto Y \pto Z$
    is $(-1)$-truncated, and thus contractible, since any pointed mapping space is inhabited.
\end{proof}

\subsection{Abelianization}\label{ss:abelianization}

In this section, we introduce the notion of abelianization, and give an
efficient construction of the abelianization of a group.

\begin{defn}
    Given a group $G$, an \define{abelianization} of $G$ consists of an abelian group $A$ together with
    a homomorphism $\eta : G \Gto A$, initial among homomorphisms to abelian groups.
    In other words, for each abelian group $B$ and homomorphism $h : G \Gto B$,
    the type $\sm{f : A \to B} h = f \circ \eta$ is contractible.
\end{defn}

Since the type of abelianizations of a given group is a mere proposition,
we abuse notation and denote any such abelianization by $G \to G^{ab}$.

\begin{rmk}
The existence of abelianizations can be proved in several different ways.
One could mimic the classical definition, describing $G^{ab}$ as the quotient
of $G$ by the subgroup generated by commutators, but this is awkward to work
with constructively.

A second method that clearly works is to define $G^{ab}$ as a higher
inductive type with a point constructor $\eta : G \to G^{ab}$,
a point constructor giving $G^{ab}$ an identity element,
recursive point constructors giving addition and inverses in $G^{ab}$,
recursive path constructors showing that the group laws hold
and that the operation is abelian,
a path constructor showing that $\eta$ is a homomorphism,
and a recursive path constructor forcing $G^{ab}$ to be a set.
While there is no doubt that this will work,
it is difficult to use in practice because of the number of constructors
and the fact that many of them are recursive.

A much simpler construction is as the higher inductive type with the
following constructors:
\begin{itemize}
\item $\eta : G \to G^{ab}$.
\item $\displaystyle \mathsf{comm} : \prod_{a, b, c : G} \, \eta (a \cdot (b \cdot c)) = \eta (a \cdot (c \cdot b))$.
\item $\displaystyle \mathsf{isset} : \prod_{x, y : G^{ab}} \, \prod_{p, q : x = y} \, p = q$.
\end{itemize}
Equivalently, this is the $0$-truncation of the coequalizer of the two
obvious maps $G \times G \times G \to G$.
Using either description, it is straightforward to show that $G^{ab}$ has a
unique group structure making $\eta$ a group homomorphism, that this group
structure is abelian, and that $\eta$ satisfies the universal property.
We don't give further details here,
since this has been formalized by Ali Caglayan in the HoTT library~\cite{HoTTCoq}.
\end{rmk}

Given a group homomorphism $f : G \Gto H$, there is a unique group
homomorphism $f^{ab} : G^{ab} \Gto H^{ab}$ making the square
\[
\begin{tikzcd}
  G \ar[d,"\eta",swap] \ar[r,"f"] & H \ar[d,"\eta"] \\
  G^{ab} \ar[r,"f^{ab}"] & H^{ab}
\end{tikzcd}
\]
commute.  This makes abelianization into a functor and $\eta$
into a natural transformation.

%
%

\subsection{Tensor products}\label{ss:tensors}
In this section, we define tensor products and use them to complete
the proof of \cref{thmS}.

Recall that for a group $G$ and an abelian group $H$, the set $G \Gto H$ is an abelian group. The group operation is given by
$(\phi + \psi)(g) \defeq \phi(g) + \psi(g)$, and the inverse by $(-\psi)(g) \defeq -\psi(g)$, along with the natural proofs that these are homomorphisms.

\begin{defn}\label{defn:tensor-product}
    Given abelian groups $A$ and $B$, a \define{tensor product} of $A$ and $B$ consists of an abelian group
    $T$ together with a map $t : A \Gto B \Gto T$ such that for any abelian group $C$ the map
    \[
        t^* : (T \Gto C) \lra (A \Gto B \Gto C)
    \]
    given by composition with $t$ is an equivalence.
\end{defn}

One can show in a straightforward way that tensor products exist, although
we don't need this, and in fact the existence follows from \cref{thm:smashistensor}.
Moreover, the type of tensor products of a given pair of abelian groups is a mere proposition.
We denote any such tensor product by $A \otimes B$.
Given $a : A$, and $b : B$, we form the \define{elementary tensor} $a \otimes b : A \otimes B$
as $a \otimes b \defeq t(a, b)$.

\begin{eg}\label{tensoringwithZ}
    Let $A : \Ab$. Then $A \simeq A \otimes \Z$, and the isomorphism is given by mapping
    $a : A$ to $a \otimes 1$.
    This follows from the fact that $\Z$ represents the identity, that is,
    $(\Z \to_\Grp C) \simeq_\Grp C$ for any $C : \Ab$, where the isomorphism
    is given by mapping $f : \Z \to_\Grp C$ to $f(1)$.
\end{eg}

\begin{lem}\label{equalonelementarytensors}
    Let $A,B,C : \Ab$, and $\phi,\psi : A \otimes B \to_{\Grp} C$.
    If for every $a : A$ and $b : B$ we have $\phi(a\otimes b) = \psi(a \otimes b)$,
    then $\phi = \psi$.
\end{lem}
\begin{proof}
    By construction, we have $\phi(a\otimes b) = t^*(\phi)(a,b)$ and $\psi(a\otimes b) = t^*(\psi)(a,b)$.
    By assumption and function extensionality, we have $t^*(\phi) = t^*(\psi)$, and
    since $t^*$ is an equivalence, we deduce that $\phi = \psi$.
\end{proof}


A key step towards proving the Hurewicz theorem is constructing a map
$\pi_n(X)^{ab} \otimes \pi_m(Y)^{ab} \Gto \pi_{n+m}(X \wedge Y)$
natural in the pointed types $X$ and $Y$, and proving that this map is an equivalence
under connectivity assumptions on $X$ and $Y$.
Equivalently, we are looking for a map
$\pi_n(X)^{ab} \Gto \pi_m(Y)^{ab} \Gto \pi_{n+m}(X \wedge Y)$
that is a tensor product under these assumptions.

In order to do this, observe that, for $G$ and $H$ groups and $A$ an abelian group, we have an equivalence
\[
    (G^{ab} \Gto H^{ab} \Gto A) \xrightarrow{\sim} (G \Gto H \Gto A),
\]
given by precomposition with the corresponding abelianization maps.
Applying the $\sma$ map from \cref{defn:sm}
to the map $\sminsbt : X \to_\sbt Y \to_\sbt X \wedge Y$ from \cref{defn:sminsbt} and
using the above observation, we get a natural map
\[
    t_{X,Y} : \pi_n(X)^{ab} \Gto \pi_m(Y)^{ab} \Gto \pi_{n+m}(X \wedge Y).
\]

\begin{thm}\label{thm:smashistensor}
    Let $n, m \geq 1$, let $X$ be a pointed, $(n-1)$-connected type, and let $Y$ be a pointed, $(m-1)$-connected type.
    Then the map $t_{X,Y}$ exhibits $\pi_{n+m}(X \wedge Y)$ as the tensor product
    of $\pi_n(X)^{ab}$ and $\pi_m(Y)^{ab}$.
\end{thm}

This implies in particular that tensor products of abelian groups exist.

\begin{proof}
    Given an abelian group $C$, we must show that the map
    \[
        t_{X,Y}^* : \left(\pi_{n+m}(\XY) \Gto C\right) \lra \left(\pi_n(X)^{ab} \Gto \pi_m(Y)^{ab} \Gto C \right) .
    \]
    is an equivalence.
    The following diagram will let us show that $t_{X,Y}^*$ is homotopic
    to a map that is easily proven to be an equivalence.
    Let $h : \pi_{n+m}(\XY) \Gto C$ and consider the diagram:
    \[
    \begin{tikzcd}
    & \left( \pi_{n+m}(\XY) \Gto C \right) \\
      \left( \XY \pto \XY \right) \ar[r,"h'_*"] \ar[d,"\sim"]
    & \left( \XY \pto K(C, n+m) \right) \ar[u,"\pi_{n+m}"',"\sim"] \ar[d,"\sim"'] \\
      \left( X \pto Y \pto \XY \right)  \ar[d,"\sma"'] \ar[r, "h'_*"]
    & \left( X \pto Y \pto K(C, n+m) \right)  \ar[d,"\sma","\sim"'] \\
      \left( \pi_n(X) \Gto \pi_m(Y) \Gto \pi_{n+m}(\XY) \right) \ar[r,"h_*"]
    & \left( \pi_n(X) \Gto \pi_m(Y) \Gto C \right)\\
      \left( \pi_n(X)^{ab} \Gto \pi_m(Y)^{ab} \Gto \pi_{n+m}(\XY) \right) \ar[r,"h_*"] \ar[u,"\sim"']
    & \left( \pi_n(X)^{ab} \Gto \pi_m(Y)^{ab} \Gto C \right) \ar[u,"\sim"] 
    \end{tikzcd}
    \]
    We explain the diagram.
    The right-hand vertical arrow at the top is an equivalence by
    \cref{cor:smashconnected,prop:connectedtotruncated}, and also implicitly
    uses a chosen equivalence $e : \pi_{n+m}(K(C, n+m)) \simeq C$.
    The unlabeled vertical arrows bordering the first square are the adjunction
    from \cref{lem:smashhomadjunction}.
    The vertical arrows labelled $\sma$ are from \cref{defn:sm};
    the right-hand one uses $e$ and is an equivalence by \cref{thm:smashisequiv}.
    The unlabeled vertical arrows at the bottom are from the universal property
    of abelianization.
    The horizontal maps labelled $h_*$ are postcomposition by $h$.
    The horizontal maps labelled $h'_*$ are postcomposition with the map
    $h' : \XY \pto K(C, n+m)$ which corresponds to $h$ under the 
    displayed equivalence $\pi_{n+m}$.
    It is straightforward to check that the three squares commute.

    The right-hand column is an equivalence which we will show is
    homotopic to $t_{X,Y}^*$.
    Consider the identity map $\idfunc{\XY}$ at the top
    of the left-hand side.  Its image in the bottom left corner is $t_{X,Y}$,
    and the image of $t_{X,Y}$ under $h_*$ is equal to the image of $h$ under $t_{X,Y}^*$.
    By definition of $h'$, the image of $\idfunc{\XY}$ in the top-right
    corner is $h$.  So the right-hand column sends $h$ to $t_{X,Y}^*(h)$.
    That is, the composite vertical equivalence is homotopic to $t_{X,Y}^*$.
\end{proof}

\subsection{Smash products, truncation, and suspension}\label{ss:smashtrunc}

The goal of this section is to prove a result about the interaction
of smash products and truncation, and a result about the interaction
of smash products and suspension. Both results make use of the symmetry
of the smash product, so we begin with that.

\begin{defn}\label{defn:tau}
    Given pointed types $X$ and $Y$, there is a pointed map $\tau : X \wedge Y \pto Y \wedge X$ defined
    by induction on the smash product in the following way:
\begin{itemize}
\item $\tau(\smin (x,y)) \defeq \smin (y,x)$;
\item $\tau(\ptl) \defeq \ptr$;
\item $\tau(\ptr) \defeq \ptl$;
\item $\ap{\tau}(\gluel\ y) \defpath \gluer\ y$.
\item $\ap{\tau}(\gluer\ x) \defpath \gluel\ x$.
\end{itemize}
It is pointed by $\refl{\smin(y_0,x_0)}$.
\end{defn}

\begin{lem}\label{lem:tauisequiv}
    For pointed types $X$ and $Y$, the composite $\tau \circ \tau : X \wedge Y \pto X \wedge Y$
    is pointed homotopic to the identity.
    In particular, the map $\tau$ is an equivalence.
\end{lem}

\begin{proof}
We first show that for every $z: X \wedge Y$, $\tau(\tau(z)) = z$.
We prove this using the induction principle for smash products.
For the three point-constructors, this holds definitionally.
The two 1-dimensional constructors are similar, so we only consider the first one.
We must show that for each $y : Y$,
\[
  \transfib{z \mapsto \tau(\tau(z)) = z}{\gluel\ y}{\refl{\smin(x_0,y)}} = \refl{\ptl} .
\]
By a calculation similar to those in~\cite[Section~2.11]{hottbook}, the left-hand-side
is equal to
\[
  \ap{\tau}(\ap{\tau}(\gluel\ y))^{-1} \cdot \refl{\smin(x_0,y)} \cdot \gluel\ y .
\]
By the definition of $\tau$ in \cref{defn:tau}, this is equal to
\[
  (\gluel\ y)^{-1} \cdot \refl{\smin(x_0,y)} \cdot \gluel\ y ,
\]
which is equal to $\refl{\ptl}$, as required.

We must also show that this homotopy is pointed.  Up to definitional equality,
this amounts to showing that $\refl{\smin(x_0,y_0)} = \refl{\smin(x_0,y_0)}$,
which is true by reflexivity.
\end{proof}

Next we show that the map $\tau$ is natural.

\begin{lem}\label{lem:taunaturality}
    Given pointed maps $f : X \pto X'$ and $g : Y \pto Y'$ between pointed types,
    the following square of pointed maps commutes:
    \[
        \begin{tikzpicture}
          \matrix (m) [matrix of math nodes,row sep=2em,column sep=3em,minimum width=2em,nodes={text height=1.75ex,text depth=0.25ex}]
          { X \wedge Y & X' \wedge Y' \\
            Y \wedge X  & Y' \wedge X'. \\};
          \path[-stealth]
            (m-1-1) edge node [above] {$f \wedge g$} (m-1-2)
                    edge node [left] {$\tau$} (m-2-1)
            (m-2-1) edge node [above] {$g \wedge f$} (m-2-2)
            (m-1-2) edge node [right]{$\tau$} (m-2-2)
            ;
        \end{tikzpicture}
    \]
\end{lem}

\begin{proof}
By path induction we can reduce to the case that
$f(x_0) \equiv x_0'$ and $g(y_0) \equiv y_0'$.
Next we use the induction principle for $X \wedge Y$.
The square commutes definitionally on the three point constructors of $X \wedge Y$,
but requires some straightforward path algebra in the remaining two cases.
Since the proof has been formalized, we omit the details.
\end{proof}

\begin{lem}\label{lem:smashtruncated}
    Let $m \geq -1$, let $n \geq 0$, let $Y$ be a pointed type, and let $X$ be a pointed, $(n-1)$-connected type.
    Then the map $\ttrunc{n+m}{|-|_m \wedge \idfunc{X}} : \ttrunc{n+m}{Y \wedge X} \to \ttrunc{n+m}{\ttrunc{m}{Y} \wedge X}$ is an equivalence.
\end{lem}

\begin{proof}
    Since the map in the statement is pointed, it is enough to show that
    for every pointed, $(n+m)$-truncated type $T$, precomposition with $|-|_m \wedge \idfunc{Y}$ induces an equivalence
    \[
        (\ttrunc{m}{Y} \wedge X \pto T) \lra (Y \wedge X \pto T).
    \]
    By the naturality in the first variable of the adjunction from \cref{lem:smashhomadjunction},
    it is enough to show that precomposition with $|-|_m$ induces an equivalence
    \[
        (\ttrunc{m}{Y} \pto X \pto T) \lra (Y \pto X \pto T),
    \]
    and this follows from the fact that the type $X \pto T$ is $m$-truncated (\cref{lem:mappingspacetrunc}).
\end{proof}

\begin{cor}\label{cor:smashtruncated}
    Let $m \geq -1$, let $n \geq 0$, let $X$ be a pointed, $(n-1)$-connected type, and let $Y$ be a pointed type.
    Then the map $\ttrunc{n+m}{\idfunc{X} \wedge |-|_m} : \ttrunc{n+m}{X \wedge Y} \to \ttrunc{n+m}{X \wedge \ttrunc{m}{Y}}$ is an equivalence.
\end{cor}

\begin{proof}
    The square
    \[
        \begin{tikzpicture}
          \matrix (m) [matrix of math nodes,row sep=2em,column sep=4em,minimum width=2em,nodes={text height=1.75ex,text depth=0.25ex}]
          { X \wedge Y & X \wedge \ttrunc{m}{Y} \\
            Y \wedge X  & \ttrunc{m}{Y} \wedge X. \\};
          \path[-stealth]
            (m-1-1) edge node [above] {${\scriptstyle\idfunc{X} \wedge |-|_m}$} (m-1-2)
                    edge node [left] {$\tau$} (m-2-1)
            (m-2-1) edge node [above] {${\scriptstyle |-|_m \wedge \idfunc{X}}$} (m-2-2)
            (m-1-2) edge node [right]{$\tau$} (m-2-2)
            ;
        \end{tikzpicture}
    \]
    commutes by \cref{lem:taunaturality}.
    The vertical maps are equivalences by \cref{lem:tauisequiv}.
    By \cref{lem:smashtruncated}, the bottom map is an equivalence after $(n+m)$-truncation,
    so the top map must also be an equivalence after truncating.
\end{proof}

We conclude this section with a result letting us commute suspension and smash products.

\begin{lem}\label{lem:suspensionandsmash}
    Given pointed types $X$ and $Y$, there is a pointed equivalence
\[
  c_\suspsym : \susp (X \wedge Y) \simeq_\sbt X \wedge \susp Y,
\]
    natural in both $X$ and $Y$.
\end{lem}

\begin{proof}
    By \cref{defn:tau} and \cref{lem:tauisequiv,lem:taunaturality}, it is enough construct a natural equivalence
    $\susp (X \wedge Y) \simeq_\sbt \susp X \wedge Y$.
    In order to do this, it suffices to show that, for every pointed type $Z$, there is an
    equivalence $(\susp (X \wedge Y) \to_\sbt Z) \simeq (\susp X \wedge Y \to_\sbt Z)$
    natural in $X$, $Y$, and $Z$, by the Yoneda Lemma (\cref{1-coherent-yoneda-embedding}).
    Given a pointed type $Z$, we define the equivalence as the following composite
    of natural equivalences:
    \begin{align*}
      (\susp (X \wedge Y) \to_\sbt Z) &\simeq (X \wedge Y \to_\sbt \loopspacesym Z)\\
                                      &\simeq (X \to_\sbt Y \to_\sbt \loopspacesym Z)\\
                                      &\simeq (X \to_\sbt \loopspacesym (Y \to_\sbt Z))\\
                                      &\simeq (\susp X \to_\sbt Y \to_\sbt Z)\\
                                      &\simeq (\susp X \wedge Y \to_\sbt Z).
    \end{align*}
    The first and fourth equivalences follow from the adjunction between suspension and loops
    (\cref{susp-loops-adjunction}).
    The second and fifth equivalences use \cref{lem:smashhomadjunction}.
    The third equivalence follows from \cref{lem:homotopygroupmappingspace}.
    This concludes the proof.
\end{proof}

This result was formalized in the Spectral repository~\cite{doornthesis},
but the proof of naturality is not complete.

\section{Homology and the Hurewicz theorem}\label{se:hurewicz}

In this section, we begin by defining homology and proving the Hurewicz theorem.
Then we define the Hurewicz homomorphism and prove that it is unique up to sign.
We conclude by giving some applications about the interaction between homology, connectedness,
and truncation.

\subsection{Prespectra and homology}\label{ss:prespectra-homology}

In this section, we introduce prespectra as a tool for defining the homology
groups of a type.

\begin{defn}
    A \define{prespectrum} $(Y,s)$ is a family of pointed types $Y : \N \to \UU_\sbt$
    and a family of pointed \define{structure maps} $s : \prd{n : \N} Y_n \pto \loopspacesym Y_{n+1}$.
    When the structure maps of $Y$ are clear from the context, we will denote the prespectrum simply by $Y$.
\end{defn}

\begin{defn}
    A \define{map of prespectra} $f : (T,s) \to (T',s')$ consists of
    a family of pointed maps $f : \prd{n : \N} Y_n \pto Y'_n$, and a family of
    pointed homotopies $\prd{n : \N} \loopspacesym s'_n \circ f_n \sim_\sbt \loopspacesym f_{n+1} \circ s_n$.
\end{defn}

By \cref{susp-loops-adjunction}, a prespectrum can be equivalently defined by giving a family of pointed types
$Y : \N \to \UU_\sbt$ and a family of pointed maps $\suspsym Y_n \pto Y_{n+1}$.
This is the way that we will specify prespectra.

\begin{eg}
    Eilenberg--Mac Lane spaces are defined in homotopy type theory in \cite{FinsterLicata}.
    Given an abelian group $A$, the \define{Eilenberg--Mac Lane prespectrum} $HA$ of type $A$
    is given by the family $\lambda n. K(A,n)$ of pointed types,
    where we let $K(A,0) \defeq A$, pointed at $0$.
    For $n \geq 1$, the structure map is
    \[
    |-|_{n+1} : \susp K(A,n) \lra \ttrunc{n+1}{\susp K(A,n)} \equiv K(A,n+1).
    \]
    When $n\equiv 0$, we
    define $\susp K(A,0) \to K(A,1)$ by induction on suspension, by mapping the north and south
    poles of $\susp K(A,0)$ to the base point of $K(A,1)$, and $\merid(a)$ to the loop
    of $K(A,1)$ represented by $a$.
\end{eg}

\begin{defn}
    Given a pointed type $X$ and a prespectrum $(Y,s)$, we form a prespectrum $X \wedge Y$,
    called the \define{smash product} of $X$ and $Y$, as follows.
    The type family is given by $(X \wedge Y)_n \equiv X \wedge Y_n$.
    The structure maps are given by the following composite:
    \begin{equation}\label{equation:structuremap}
        \susp (X \wedge Y_n) \xrightarrow{c_\suspsym}_\sbt
        X \wedge \susp Y_n \xrightarrow{\idfunc{X} \wedge \overline{s_n}}_\sbt
        X \wedge Y_{n+1},
    \end{equation}
    where $c_{\Sigma}$ is the map from \cref{lem:suspensionandsmash}
    and $\overline{s_n} : \susp Y_n \pto Y_{n+1}$ corresponds to $s_n : Y_n \pto \loopspacesym Y_{n+1}$.
\end{defn}

Note that, by the naturality of \cref{lem:suspensionandsmash}
and the functoriality of the smash product on pointed types,
the smash product of a pointed type and a prespectrum is functorial.

\begin{defn}
    The type of \define{sequential diagrams of groups} is the type
    \[
        \Grp^\N \defeq \sum_{A : \N \to \Grp} \, \prod_{n : \N} \, A_n \to_\Grp A_{n+1}.
    \]

    Analogously, we define the type of \define{sequential diagrams of abelian groups}, which
    we denote by $\Ab^\N$.
\end{defn}

The most important example in this paper is given by sequential diagrams of
groups that come from prespectra.

\begin{eg}
    Let $(Y,s)$ be a prespectrum and let $n,k : \N$.
    The map $s_k : Y_k \to_\sbt \loopspacesym Y_{k+1}$ induces a morphism
    $\pi_n(s_k) : \pi_n(Y_k) \to_{\Grp} \pi_n(\loopspacesym Y_{k+1}) \simeq \pi_{n+1}(Y_{k+1})$.
    Iterating this process, we get a sequential diagram of groups
    $\lambda i. \pi_{n+i}(Y_{k+i}) : \N \to \Grp$.
    We denote this diagram by $\stab_k^n(Y)$.
    This construction is natural in $Y$.

    Note that, if $n\geq 2$, the diagram $\stab_k^n(Y)$ is a sequential diagram of abelian groups.
\end{eg}

\begin{defn}
    Let $(A,\phi) : \UU^\N$ be a sequential diagram of types,
    We define the \define{sequential colimit} of $(A,\phi)$, denoted by $\colim A : \UU$,
    as the higher inductive type generated by
    the constructors $\iota : \prod_{n : \N} A_n \to \colim A$
    and $\glue : \prod_{n : \N} \prod_{a : A_n} \iota_n(a) = \iota_{n+1}(\phi_n(a))$.
\end{defn}

\begin{lem}\label{lem:seqcolimofgroups}
    Let $(A,\phi) : \Ab^\N$ be a sequential diagram of abelian groups.
    Then, the sequential colimit $\colim A$ of the underlying sets is a set, and it has a
    canonical abelian group structure such that all of the induced maps
    $i_n : A_n \to \colim A$ are homomorphisms.
    Moreover, the abelian group $\colim A$ has the universal property of
    the colimit in the category of abelian groups.
\end{lem}

\begin{proof}
    The main difficulty is to show that $\colim A$ is $0$-truncated.
    For this, we use \cite[Corollary~7.7~(1)]{DoornRijkeSojakova}, which says that
    a sequential colimit of $n$-truncated types is $n$-truncated.

    To show that $\colim A$ has an abelian group structure we start by using
    induction to define the operation $+$ on $\colim A$.
    In the case of point constructors, we define
    \mbox{$\iota_l(a) + \iota_n(b) \defeq \iota_m(\phi_l^m(a) + \phi_n^m(b))$,}
    where $m \equiv \max(l,n)$ and $\phi_l^m : A_l \to A_m$ and
    $\phi_n^m : A_n \to A_m$ are defined by iterating $\phi$.
    The case of a path constructor $\glue$ and a point constructor is
    straightforward, and the case of two path constructors is immediate, since $\colim A$ is a set.
    The fact that, with these operation, $\colim A$ is an abelian group is clear.

    The map $\iota_n : A_n \to \colim A$ is a group morphism for every $n$
    by construction, and the fact that $\colim A$ satisfies the universal
    property of the colimit follows from the induction principle of $\colim A$.
\end{proof}

\begin{defn}\label{defn:stablehomotopygroups}
    Let $Y$ be a prespectrum, let $n : \Z$, and let $j \equiv \max(0, 2-n)$.
    We define the \define{$n$-th stable homotopy group} of $Y$ as
    \[
        \pi^s_n(Y) \defeq \colim \stab^{n+j}_j(Y).
    \]
\end{defn}

Note that the stable homotopy groups of a prespectrum are
defined for any integer $n$, and not only for natural numbers.
Moreover, by construction, the sequential diagram in the definition of $\pi^s_n(Y)$
is a sequential diagram of abelian groups, so stable homotopy
groups are always abelian.
As an aside, one can show that any $\stab^{n+j}_j(Y)$ with $j \geq \max(0, 2-n)$
will have an isomorphic colimit.
Finally, since the construction $\stab^n_k(Y)$ is functorial in $Y$, stable homotopy
groups are functorial in the prespectrum.

\begin{defn}\label{defn:homology}
    We define the $n$-th \define{reduced homology} of $X$ with coefficients in $Y$ as
    \[
        \tH_n(X; Y) \defeq \pi_n^s(X \wedge Y).
    \]
    We define the $n$-th \define{(ordinary) reduced homology} of $X$ with coefficients in
    an abelian group $A$ by
    \[
        \tH_n(X; A) \defeq \tH_n(X; HA).
    \]
\end{defn}

Notice that these types carry an abelian group structure, given by the group structure of stable homotopy groups
(\cref{defn:stablehomotopygroups}).

\subsection{The Hurewicz theorem}\label{ss:hurewicz}

In this section, we prove our main result, \cref{thmH}.
To do so, we first show that when $X$ is sufficiently connected,
we can compute $\tH_n(X; A)$ without taking a colimit.

\begin{lem}\label{lem:stabilization}
    Let $n\geq 1$, let $A : \Ab$, and let $X$ be a pointed, $(n-1)$-connected type.
    Then the natural homomorphism $\pi_{n+1}(X \wedge K(A,1)) \to \tH_n(X; A)$
    is an equivalence.
\end{lem}

\begin{proof}
    Recall that $\tH_n(X;A) \equiv \pi_n^s(X \wedge HA) \equiv \colim \stab^{n+j}_j(X \wedge HA)$, for
    $j = \max(0,2-n)$.
    Since $n \geq 1$, we must consider two cases, $n=1$ and $n \geq 2$.
    When $n = 1$, we have $j = 1$, and
    the sequential diagram that defines $\tH_n(X; A)$ starts as follows:
    \[
        \phantom{\pi_n(X \wedge K(A,0)) \lra} \pi_{n+1}(X \wedge K(A,1)) \lra \pi_{n+2}(X \wedge K(A,2)) \lra \cdots
    \]
    When $n \geq 2$, we have $j = 0$, and
    the sequential diagram that defines $\tH_n(X; A)$ starts as follows:
    \[
        \pi_n(X \wedge K(A,0)) \lra \pi_{n+1}(X \wedge K(A,1)) \lra \pi_{n+2}(X \wedge K(A,2)) \lra \cdots
    \]
    It suffices to show that in either case the morphism
    \[
        \pi_{n+i}(X \wedge K(A,i)) \lra \pi_{n+i+1}(X \wedge K(A,i+1))
    \]
    is an equivalence for $i \geq 1$.
    To prove this, we use \cref{equation:structuremap} to factor the map as
    \[
        \pi_{n+i}(X \wedge K(A,i)) \lra \pi_{n+i+1}(X \wedge \suspsym K(A,i)) \lra \pi_{n+i+1}(X \wedge K(A,i+1)).
    \]
    Now, the first of these two maps is induced by the Freundenthal map
    $X \wedge K(A,i) \to \loopspacesym \suspsym (X \wedge K(A,i))$ composed with the equivalence
    $\suspsym (X \wedge K(A,i)) \simeq (X \wedge \suspsym K(A,i))$.
    Notice that, by \cref{cor:smashconnected}, $X \wedge K(A,i)$ is $(n+i-1)$-connected.
    If $i\geq 1$, we have that $n+i \geq 2$, and thus $(n+i-1)+1 \leq 2(n+i-1)$,
    so the Freudenthal suspension theorem (\cite[Theorem~8.6.4]{hottbook}) implies that
    the map $\pi_{n+i}(X \wedge K(A,i)) \to \pi_{n+i+1}(X \wedge \suspsym K(A,i))$
    is an equivalence.

    The second map is an equivalence by \cref{cor:smashtruncated}, since by definition $K(A,i+1) \equiv \ttrunc{i+1}{\suspsym K(A,i)}$.
\end{proof}

\begin{thm}[Hurewicz Theorem]\label{thm:hurewicz}
    Given an abelian group $A$, a natural number $n\geq 1$, and a pointed, $(n-1)$-connected type $X$,
    we have an isomorphism $\pi_n(X)^{ab} \otimes A \simeq_\Grp \tH_n(X; A)$, natural in $X$ and $A$.
\end{thm}

By naturality in $X$, we mean naturality with respect to pointed maps between $(n-1)$-connected types.

\begin{proof}
    By \cref{lem:stabilization}, it is enough to show that we have a natural isomorphism
    $\pi_{n+1}(X \wedge K(A,1)) \simeq_\Grp \pi_n(X)^{ab} \otimes A$,
    and this follows directly from \cref{thm:smashistensor}.
\end{proof}

\subsection{The Hurewicz homomorphism}\label{ss:hurewicz-hom}

In this section we give a construction of the Hurewicz homomorphism
and prove that it is unique up to sign.

Let $X$ be a pointed type, $A$ an abelian group, and $n \geq 1$.
Applying $\tH_n(-;A)$ to the $(n-1)$-connected cover map $X\cov{n-1} \to_\sbt X$
we obtain a morphism $\tH_n(X\cov{n-1} ; A) \to_\Grp \tH_n(X ; A)$, natural
in $X$ and $A$.
By \cref{thm:hurewicz}, there is a natural isomorphism
\[
\pi_n(X\cov{n-1})^{ab} \otimes A \simeq_\Grp \tH_n(X\cov{n-1} ; A).
\]
Since $\pi_n(X\cov{n-1}) \to_\Grp \pi_n(X)$ is also a natural isomorphism, we can compose with
the abelianization of its inverse to obtain a morphism $\pi_n(X)^{ab} \otimes A \to_\Grp \tH_n(X ; A)$.

\begin{defn}\label{def:Hurewicz-homomorphism}
    For every $X : \UU_\sbt$, $A : \Ab$, and $n \geq 1$,
    the morphism $h_n : \pi_n(X)^{ab} \otimes A \to \tH_n(X ; A)$ described above is the \define{$n$th Hurewicz homomorphism}.
\end{defn}

By construction, when $X$ is $(n-1)$-connected, $h_n$ is an isomorphism.

\begin{defn}
    Let $n \geq 1$. A morphism of \define{$n$-Hurewicz type} is given by
    a group homomorphism $\pi_n(X)^{ab} \otimes A \to_{\Grp} \tH_n(X;A)$
    for each $X : \UU_\sbt$ and $A : \Ab$,
    that is natural in both $A$ and $X$,
    and that is an isomorphism when $X \equiv S^n$ and $A \equiv \Z$.
    Here we are regarding $\pi_n(-)^{ab} \otimes -$ and $\tH_n(-;-)$
    as $1$-coherent functors $\UUcat_\sbt \times \Ab \to \Ab$.
\end{defn}

%

\begin{eg}
    For any $n \geq 1$, the $n$th Hurewicz homomorphism (\cref{def:Hurewicz-homomorphism}) is a morphism
    of $n$-Hurewicz type.
\end{eg}

\begin{thm}\label{thm:Hurewicz-unique}
    Let $n : \N$ and let
    $F$ and $G$ be morphisms of $n$-Hurewicz type.
    Then either $F(X,A) = G(X,A)$ or
    $F(X,A) = -G(X,A)$ for every pointed type $X$ and
    abelian group $A$.
    The choice of sign is independent of $X$ and $A$.
\end{thm}
\begin{proof}
    The morphisms $F(S^n,\Z)$ and $G(S^n,\Z)$ give us two
    isomorphisms between $\pi_n(S^n)\otimes \Z$ and $\tH_n(S^n;\Z)$.
    We now show that there are exactly two possible isomorphisms between $\pi_n(S^n)\otimes \Z$ and $\tH_n(S^n;\Z)$,
    and that these differ by a sign.
    On the one hand, by~\cite{LicataShulman} (see also~\cite[Section~8.1]{hottbook}), we know $\pi_n(S^n) \simeq \Z$.
    On the other hand, we have $\Z \otimes \Z \simeq \Z$ (\cref{tensoringwithZ}).
    So it is enough to show that there are exactly two isomorphisms between $\Z$ and $\Z$,
    and that they differ by a sign.
    This is straightforward, using the fact that if two integers $n$ and $m$ satisfy
    $n \times m = 1$, then $n = m = 1$ or $n = m = -1$, which follows from the fact
    that $\Z$ has decidable equality.

    There are then two cases, $F(S^n,\Z) = G(S^n,\Z)$ and $F(S^n,\Z) = -G(S^n,\Z)$.
    We consider only the first case, the second one being analogous.
    We thus assume that $F(S^n,\Z) = G(S^n,\Z)$ and we want to show that
    for every pointed type $X$ and every abelian group $A$ we have
    $F(X,A) = G(X,A)$.

    By \cref{equalonelementarytensors}, it is enough to check that $F(X,A) = G(X,A)$
    when evaluated on elementary tensors.
    Since the abelianization map is surjective and we are proving a proposition,
    it is enough to check this on elementary tensors $(\eta \alpha) \otimes \beta$
    for $\alpha : \pi_n(X)$ and $\beta : A$.
    Since we are proving a mere proposition, we can assume that we have
    a pointed map $\overline{\alpha} : S^n \to_{\sbt} X$ representing $\alpha$.
    Define $\overline{\beta} : \Z \to A$ by sending $1$ to $\beta$.
    Consider the following diagram, which commutes by the naturality assumption:
    \[
        \begin{tikzpicture}
          \matrix (m) [matrix of math nodes,row sep=4em,column sep=5em,minimum width=2em,nodes={text height=1.75ex,text depth=0.25ex}]
            { \pi_n(S^n)^{ab} \otimes \Z & \tH_n(S^n; \Z) \\
            \pi_n(X)^{ab} \otimes A & \tH_n(X;A). \\};
          \path[-stealth]
            (m-1-1) edge node [above] {$F(S^n,\Z)$} (m-1-2)
                    edge node [left] {$\pi_n(\overline{\alpha})^{ab} \otimes \overline{\beta}$} (m-2-1)
            (m-2-1) edge node [above] {$F(X,A)$} (m-2-2)
            (m-1-2) edge node [right]{$\tH_n(\overline{\alpha},\overline{\beta})$} (m-2-2)
            ;
        \end{tikzpicture}
    \]
    The commutativity of the diagram implies that
\[
F(X,A)((\eta \alpha) \otimes \beta) = \tH_n(\overline{\alpha},\overline{\beta})(F(S^n,\Z)(\theta \otimes 1)) ,
\]
    where $\theta : \pi_n(S^n)$ is represented by the identity map $S^n \pto S^n$.
    Similarly, we get that $G(X,A)((\eta \alpha) \otimes \beta) = \tH_n(\overline{\alpha},\overline{\beta})(G(S^n,\Z)(\theta \otimes 1))$,
    and since $F(S^n,\Z) = G(S^n,\Z)$, we conclude that $F(X,A)(\alpha \otimes \beta) = G(X,A)(\alpha \otimes \beta)$.
\end{proof}

\subsection{Applications}\label{ss:applications}

In this section, we give some consequences of the main results in the paper.
We start with two immediate applications of the Hurewicz theorem.

\begin{prop}
    Let $n \geq 1$, let $X$ be a pointed, $n$-connected type, and let $A : \Ab$.
    Then $\tH_i(X;A) = 0$ for all $i \leq n$.
    Conversely, if $X$ is a pointed, connected type with abelian fundamental group
    such that $\tH_i(X;\Z) = 0$ for all $i \leq n$, then $X$ is $n$-connected.\qed
\end{prop}

\begin{prop}
    Let $n \geq 1$ and let $A,B : \Ab$. Then $\tH_n(K(A,n);B) \simeq A\otimes B$,
    and, in particular, $\tH_n(K(A,n);\Z) \simeq A$.\qed
\end{prop}

The following result says that truncation does not affect low-dimensional homology.

\begin{prop}\label{prop:homologyoftruncation}
    Let $X$ be a pointed type and let $m \geq n$ be natural numbers.
    For every abelian group $A : \Ab$, the truncation map $X \to \ttrunc{m}{X}$ induces
    an isomorphism $\tH_n(X;A) \xrightarrow{\sim} \tH_n\left(\ttrunc{m}{X};A\right)$.
\end{prop}
\begin{proof}
    The objects in the sequential diagram that defines
    $\tH_n(X;A)$ have the form $\pi_{n+i}(X \wedge K(A,i))$
    for $i \geq \min(0,2-n)$,
    and the morphism $\tH_n(X;A) \to \tH_n(\ttrunc{m}{X};A)$ is induced by
    levelwise morphisms $\pi_{n+i}(X \wedge K(A,i)) \to \pi_{n+i}(\ttrunc{m}{X} \wedge K(A,i))$
    given by the functoriality of $\pi_{n+i}$ and the smash product.
    We will show that these levelwise morphisms are isomorphisms,
    which implies that the induced map is an isomorphism.

    Consider the commutative square
   \[
        \begin{tikzpicture}
          \matrix (m) [matrix of math nodes,row sep=2em,column sep=3em,minimum width=2em,nodes={text height=1.75ex,text depth=0.25ex}]
          { \pi_{n+i}(X \wedge K(A,i))  & \pi_{n+i}\left(\ttrunc{m}{X} \wedge K(A,i)\right) \\
            \pi_{n+i}\left(\ttrunc{i+m}{X \wedge K(A,i)}\right)  & \pi_{n+i}\left(\ttrunc{i+m}{\ttrunc{m}{X} \wedge K(A,i)}\right), \\};
          \path[-stealth]
            (m-1-1) edge node [above] {} (m-1-2)
                    edge node [left] {} (m-2-1)
            (m-2-1) edge node [above] {} (m-2-2)
            (m-1-2) edge node [right]{} (m-2-2)
            ;
        \end{tikzpicture}
    \]
    given by functoriality of $(i+m)$-truncation and $\pi_{n+i}$.
    It suffices to show that the bottom map and the vertical maps in the square
    are isomorphisms.
    The vertical maps are isomorphisms since $n+i \leq i + m$, and the bottom
    map is an isomorphism by \cref{lem:smashtruncated}.
\end{proof}

We conclude by showing that $\infty$-connected maps induce an isomorphism in all homology groups.

\begin{cor}
    Let $f : X \to_\sbt Y$ be a pointed map between pointed types that induces an
    isomorphism in $\pi_0$ and an isomorphism in $\pi_n$ for $n \geq 1$ and all
    choices of basepoint $x_0 : X$. Then $f$ induces an isomorphism in all homology groups
    for all choices of coefficients.
\end{cor}
\begin{proof}
    Let $A : \Ab$ and let $n \geq 0$.
    We have a commutative square
   \[
        \begin{tikzpicture}
          \matrix (m) [matrix of math nodes,row sep=2em,column sep=3em,minimum width=2em,nodes={text height=1.75ex,text depth=0.25ex}]
            { \tH_n(X; A) & \tH_n(Y ; A)\\
              \tH_n\left(\ttrunc{n}{X}; A\right) & \tH_n\left(\ttrunc{n}{Y}; A\right) ,\\};
          \path[-stealth]
            (m-1-1) edge node [above] {} (m-1-2)
                    edge node [left] {} (m-2-1)
            (m-2-1) edge node [above] {} (m-2-2)
            (m-1-2) edge node [right]{} (m-2-2)
            ;
        \end{tikzpicture}
    \]
    where the vertical maps are isomorphisms, by \cref{prop:homologyoftruncation}.
    The bottom map is induced by $\ttrunc{n}{f} : \ttrunc{n}{X} \to \ttrunc{n}{Y}$,
    which is an equivalence, by the truncated Whitehead theorem (\cite[Theorem~8.8.3]{hottbook}).
    It follows that the top map is an isomorphism, concluding the proof.
\end{proof}

%

\printbibliography

\end{document}